\documentclass[11pt]{article}

\usepackage{fullpage}
\usepackage{amsmath,amsthm,amsfonts,dsfont}
\usepackage{amssymb,latexsym,graphicx}
\usepackage{palatino}
\usepackage{mathpazo}
\usepackage{stmaryrd}
\usepackage{mathtools}
\usepackage{hyperref}
\usepackage{graphicx}

\newtheorem{theorem}{Theorem}[section]

\newtheorem{lemma}[theorem]{Lemma}

\newtheorem{corollary}[theorem]{Corollary}

\newtheorem{question}{Question}[section]

\newcommand{\R}{\ensuremath{\mathbb{R}}}
\newcommand{\Z}{\ensuremath{\mathbb{Z}}}

\renewcommand{\epsilon}{\varepsilon}

\renewcommand{\vec}[1]{\ensuremath{\mathbf{#1}}}

\makeatletter
\def\imod#1{\allowbreak\mkern8mu({\operator@font mod}\,\,#1)}
\makeatother

\newcommand{\lat}{\mathcal{L}}

\DeclareMathOperator{\vol}{vol}

\DeclarePairedDelimiter\inner{\langle}{\rangle}

\DeclarePairedDelimiter\set{\{}{\}}

\DeclarePairedDelimiter\length{\lVert}{\rVert}

\begin{document}

\title{A Note on Koldobsky's Lattice Slicing Inequality}
\author{
Oded Regev\thanks{Courant Institute of Mathematical Sciences, New York
 University. Supported by the Simons Collaboration on Algorithms and Geometry and by the National Science Foundation (NSF) under Grant No.~CCF-1320188. Any opinions, findings, and conclusions or recommendations expressed in this material are those of the authors and do not necessarily reflect the views of the NSF.}
}
\date{}
\maketitle


%

\begin{abstract}
We show that if $K \subset \R^d$ is an origin-symmetric convex body, then there exists a vector $\vec{y} \in \Z^d$ such that 
\begin{align*}
|K \cap \Z^d \cap \vec{y}^\perp| / |K \cap \Z^d|
\ge 
\min(1,c \cdot d^{-1} \cdot \vol(K)^{-1/(d-1)}) \; ,
\end{align*}
for some absolute constant $c>0$, where $\vec{y}^\perp$ denotes the subspace orthogonal to $\vec{y}$. This gives a partial answer to a question by Koldobsky.
\end{abstract}

\section{Introduction}

The following question was asked by Koldobsky 
during the 2013 AIM workshop on ``Sections of convex bodies.''

\begin{question}
Is it the case that for all $d \ge 1$ there exists an $\alpha=\alpha(d)>0$ satisfying the following:
for all origin-symmetric convex bodies $K \subset \R^d$ such that ${\rm span}(K \cap \Z^d) = \R^d$
there exists a nonzero $\vec{y} \in \R^d$ such that
\[
|K \cap \Z^d \cap \vec{y}^\perp| / |K \cap \Z^d| \ge \alpha \cdot \vol_d(K)^{-1/d} \; .
\]
\end{question}
\noindent
In other words, the question asks to find a dimension $d-1$ subspace that contains at least an $\alpha \cdot \vol_d(K)^{-1/d}$ fraction of the lattice points 
in $K$. 
The requirement on the span is in order to avoid degenerate cases of bodies of very small volume that would force $\alpha$ to be small. 

Alexander, Henk, and Zvavitch~\cite{AlexanderHZ15} gave a positive answer to this question by showing that one can 
take $\alpha = C^{-d}$ for some absolute constant $C$. 
They also showed that for the special case of unconditional bodies $K$ one can take $\alpha = c/d$ for some absolute constant $c>0$,
and observed that this is tight (as follows by taking $K$ to be the cross-polytope ${\rm conv}(\pm \vec{e}_1,\ldots,\pm \vec{e}_d)$).
It remains an open question whether one can take $\alpha = c/d$ for general bodies. 
In this note we show that this is the case for bodies whose volume is at most $C^{d^2}$ for any constant $C>0$ (see Theorem~\ref{thm:main} for the full statement).
We refer to~\cite{AlexanderHZ15} for further background on Koldobsky's question and its connection to the slicing problem of Bourgain.

\subsubsection*{Acknowledgements } I am grateful to Assaf Naor for suggesting that I look at Koldobsky's question. 

\section{Orthogonal lattice points}

We use the convention that $c$ is an arbitrary absolute positive constant which might differ from one
occurrence to the next. 
A \emph{lattice} $\lat \subset \R^d$ is defined as the set of
all integer linear combinations of $d$ linearly independent vectors 
in $\R^d$.  
The \emph{dual lattice} of $\lat$ is $\lat^* = \set{\vec{y} \in
\R^d \,:\, \inner{\vec{y},\vec{x}} \in \Z, \forall \vec{x} \in \lat}$. 
For any $s>0$, we define the function $\rho_s : \R^n \rightarrow\R$ as $\rho_s(\vec{x}) = \exp(-\pi \length{\vec{x}}^2/s^2)$. For a discrete set $A \subset \R^n$ we define $\rho_s(A)=\sum_{\vec{x}\in A} \rho_s(\vec{x})$ and denote by $D_{A,s}$ the probability distribution assigning mass $\rho_s(\vec{x})/\rho_s(A)$ to any $\vec{x} \in A$. 
Recalling that the Fourier transform of $\rho_s(\cdot)$ is $s^d \rho_{1/s}(\cdot)$, the following is an immediate application of the Poisson summation formula. 

\begin{lemma}\label{lem:origin}
For any lattice $\lat \subset \R^d$ and $s>0$, 
\[
\rho_s(\lat) = (\det \lat)^{-1} s^d \cdot \rho_{1/s}(\lat^*) \; .
\]
In particular,
\[
\rho_s(\lat) \ge (\det \lat)^{-1} \cdot s^d \; ,
\]
or equivalently, 
\[
\Pr_{\vec{y} \sim D_{\lat,s}}[\vec{y}=0] \le (\det \lat) \cdot s^{-d} \; .
\]
\end{lemma}

The following is another easy corollary of the Poisson summation formula (already
used in~\cite{banaszczyk}), and holds because $\rho_s$ is a positive definite function, i.e., a 
function with a non-negative Fourier transform. 

\begin{lemma}\label{lem:translations}
For any lattice $\lat \subset \R^d$, $s>0$, and $\vec{x} \in \R^d$,
\[
\rho_s(\lat + \vec{x}) \le \rho_s(\lat) \; .
\]
\end{lemma}

\begin{corollary}\label{cor:massonzero}
For any lattice $\lat \subset \R^d$, $s>0$, and $\vec{x} \in \lat$,
\[
\Pr_{\vec{y} \sim D_{\lat^*,s}}[ \inner{\vec x,\vec y} = 0 ]
\ge 
\Pr_{y \sim D_{\Z / \|\vec x\|, s}}[y=0] 
= 
\rho_{s \|\vec x\|}(\Z)^{-1} 
\ge 
c \cdot \min(1,(s \|\vec x\|)^{-1}) \; .
\]
\end{corollary}
\begin{proof} 
We start with the first inequality. 
Clearly, $\inner{\vec x,\vec y}$ takes integer values. 
For any $i \in \Z$, 
the set of points $\vec{y}$ in $\lat^*$ with $\inner{\vec x,\vec y} = i$ is
either empty or a translation of $\lat^* \cap \vec{x}^\perp$ whose affine span is at distance
$i / \|\vec x\|$ from the origin. The $\rho_s$ mass of this set is obviously zero in the
former case and at most $\rho_s(i / \|\vec x\|) \rho_s(\lat^* \cap \vec{x}^\perp)$ in the latter
by Lemma~\ref{lem:translations} and the product property of $\rho_s$. The inequality follows.
The last inequality is an easy calculation. 
\end{proof}

%
%

\section{Application to Koldobsky's question}

\begin{theorem}\label{thm:main}
Let $K \subset \R^d$ be an origin-symmetric convex body.
Then there exists a vector $\vec{y} \in \Z^d$ such that 
\begin{align*}
|K \cap \Z^d \cap \vec{y}^\perp| / |K \cap \Z^d|
\ge 
\min(1, c \cdot d^{-1} \cdot \vol(K)^{-1/(d-1)}) \; . 
\end{align*}
\end{theorem}

We note that this bound improves on that of Alexander et al.~\cite{AlexanderHZ15} for any body whose volume is at most $c^{d^3}$.

\begin{proof}
By John's theorem (see, e.g.,~\cite{Ball97}), there exists a linear transformation $T$ of determinant $1$
so that $TK$ has circumradius at most $d \cdot \vol(K)^{1/d}$. Therefore, by considering
the body $TK$ and the determinant $1$ lattice $T\Z^d$, it suffices to prove the following: 
for any origin-symmetric convex body $K \subset \R^d$ with circumradius
at most $d \cdot \vol(K)^{1/d}$,
and any lattice $\lat \subset \R^d$, 
there exists a nonzero vector $\vec{y} \in \lat^*$ such that 
\begin{align*}
|K \cap \lat \cap \vec{y}^\perp| / |K \cap \lat|
\ge 
\min(1, c \cdot d^{-1} \cdot \vol(K)^{-1/(d-1)}) \; . 
\end{align*}

Notice that if $\vol(K) < d^{-d}$ then the circumradius of $K$ is less than $1$, 
in which case $\lat \cap K$ is not full dimensional (by Hadamard's inequality and $\det \lat = 1$).
As a result, we can choose a vector $\vec{y}$ 
so that $K \cap \lat \subset \vec{y}^\perp$, making the quotient above $1$.
We therefore assume from now on that $\vol(K) \ge d^{-d}$. 

We will use a simple application of the probabilistic method. 
Namely, let us choose $\vec{y}$ from the distribution 
$D_{\lat^*,s}$ where $s=C \vol(K)^{1/(d(d-1))} \ge 1$
for a large enough absolute constant $C>0$.
Then, by Corollary~\ref{cor:massonzero}, for any 
fixed $\vec{x} \in K \cap \lat$, the probability
that $\inner{\vec{x},\vec{y}}=0$ is at least 
\[
c \cdot (s d \cdot \vol(K)^{1/d})^{-1} =
c \cdot d^{-1} \cdot \vol(K)^{-1/(d-1)} \;. 
\]
Let $p$ denote the latter quantity. 
It follows that the expected fraction of vectors $\vec{x}$ in 
$K \cap \lat$ satisfying $\inner{\vec{x},\vec{y}}=0$
is at least $p$. 
Moreover, by Markov's inequality, with probability at least
$p/2$ over the choice of $\vec{y}$, the fraction of vectors
$\vec{x}$ in $K \cap \lat$ satisfying $\inner{\vec{x},\vec{y}}=0$
is at least $p/2$. To complete the proof, notice
by Lemma~\ref{lem:origin} that the probability
that $\vec{y}=0$ is $s^{-d} < p/2$. Therefore,
there is a positive probability over the choice of $\vec{y}$
that $\vec{y}\neq 0$ and moreover, 
that the fraction of vectors
$\vec{x}$ in $K \cap \lat$ satisfying $\inner{\vec{x},\vec{y}}=0$
is at least $p/2$. This completes the proof. 
\end{proof}

\bibliographystyle{alpha}
\bibliography{koldobsky}

\end{document}